\documentclass[12pt, a4paper]{amsart}
\usepackage[english]{babel}
\usepackage{amssymb,amsmath,amsthm,mathrsfs,bm}
\usepackage{color}
\usepackage[margin=1in]{geometry}
\usepackage{stmaryrd}

\newtheorem{theorem}[subsection]{Theorem}

\newtheorem{definition}[subsection]{Definition}
\newtheorem{lemma}[subsection]{Lemma}
\newtheorem{remark}[subsection]{Remark}
\newtheorem{proposition}[subsection]{Proposition}
\newtheorem{corollary}[subsection]{Corollary}

\newtheorem*{claim*}{Claim}
\newtheorem*{theorem*}{Theorem}

\def\bal{\begin{aligned}}
\def\eal{\end{aligned}}
\def\be{\begin{equation}\label}
\def\ee{\end{equation}}
\def\bcs{\begin{cases}}
\def\ecs{\end{cases}}
\def\={\;=\;}
\def\+{\,+\,}
\def\-{\,-\,}

\def\Z{{\mathbb Z}}

\def\R{{\mathbb R}}
\def\F{{\mathbb F}}

\def\lb{\llbracket}
\def\rb{\rrbracket}
\def\ord{\mathrm{ord}}

\def\sR{\mathcal{R}}
\def\sS{\mathcal{S}}
\newcommand*\Bell{\ensuremath{\boldsymbol\ell}}

\def\cartier{\mathscr{C}_p}

\def\v#1{{\bf #1}}

\def\is{\equiv}
\def\mod#1{({\rm mod}\ #1)}

\title{Dwork crystals II}
\author{Frits Beukers, Masha Vlasenko}

\address{Utrecht University }
\email{f.beukers@uu.nl}

\address{Institute of Mathematics of the Polish Academy of Sciences}
\email{m.vlasenko@impan.pl}

\thanks{
Work of Frits Beukers was supported by the Netherlands Organisation for Scientific Research (NWO), grant TOP1EW.15.313. Work of Masha Vlasenko was supported by the National Science Centre of Poland (NCN), grant UMO-2016/21/B/ST1/03084.}

\begin{document}
\maketitle

\section{Introduction}

This paper is a continuation of~\cite{DC-I}, which we will refer to as Part I. 

In Part I we considered $p$-adic limit formulas to matrices of the so-called
Cartier action. As an example, consider the elliptic curve $f(x,y)=y^2-x(x-1)(x-z)=0$.
Let $G_m(z)$ be the coefficient of $(xy)^{m-1}$ in $f(x,y)^{m-1}$. Let $z_0\in\Z_p$ and we denote its residue modulo $p$ by $\overline{z}_0 \in \F_p$. Then it was
shown in Part I that, if $G_p(\overline{z}_0) \ne 0$ the quotients $G_{p^s}(z_0)/G_{p^{s-1}}(z_0)$
form a $p$-adic Cauchy sequence tending to the unit root $\lambda(\overline{z}_0) \in \Z_p^\times$  of the zeta function of
$y^2=x(x-1)(x-\overline{z}_0)$ as $s\to\infty$. Furthermore, when $z$ is a variable, the quotients
$G_{p^s}(z)/G_{p^{s-1}}(z^p)$ form a $p$-adic Cauchy sequence as $s\to\infty$. The limit of this sequence can be identified as $(-1)^{\frac{p-1}2} F(z)/F(z^p)$, where $F(z)$ denotes the hypergeometric function $F(1/2,1/2,1|z) = \sum_{k=0}^{\infty}\frac{(1/2)_k^2}{k!^2}z^k$. This computation was done in Example~5.5 of Part I. It then follows from the results in Part~I that the ratio $F(z)/F(z^p) \in \Z_p\lb z\rb$ can be approximated $p$-adically by rational functions whose denominators are powers of $G_p(z)$. This property was observed earlier by Bernard Dwork, who used a different kind of $p$-adic approximation,~\cite{dwork63,dwork69}. In this particular case, we can show (Remark~\ref{dworkremark1}) that 
\be{DC-original}
F(z)/F(z^p) \equiv F_{p^s}(z) / F_{p^{s-1}}(z^p) \; \mod {p^s},
\ee 
where $F_m(z)=\sum_{k=0}^{m-1}\frac{(1/2)_k^2}{k!^2}z^k$ are truncations of $F(z)$. This congruence is a version of \cite[(12)]{dwork63}.

Here $F_p(z) \equiv G_p(z) \mod p$. We will also see that if $z_0\in\Z_p$ and $G_p(\overline{z}_0) \ne 0$, the sequence $(-1)^{\frac{p-1}2} F_{p^s}(z_0)/F_{p^{s-1}}(z_0)$
tends to  the unit root $\lambda(\overline{z}_0)$. This is clarified in Remark \ref{dworkremark2}.

In this paper we will give a vast generalization and explain the underlying mechanism of congruences of the above type. For a generic Laurent polynomial $f$, it turns out that the corresponding
generalization of $F(z),F_m(z)$ is given by A-hypergeometric series and their
truncations.

We now recall the notations and definitions from Part I.

Let $p$ be a prime and $R$ a $p$-adically complete characteristic zero domain such that $\cap_s p^s R=\{ 0 \}$. Let $f \in R[x_1^{\pm1},\ldots,x_n^{\pm 1}]$ be a Laurent polynomial and $\Delta \subset \R^n$ be its Newton polytope. A subset $\mu \subset \Delta$ is said to be \emph{open} if its complement $\Delta \setminus \mu$ is a union of faces of any dimensions. For such a subset we consider the $R$-module of rational functions
\[
\Omega_f(\mu) = \left\{ (k-1)!\frac{g(\v x)}{f(\v x)^k} \;\Big|\; k \ge 1, g \in R[x_1^{\pm1},\ldots,x_n^{\pm 1}],{\rm supp}(g) \subset k \mu  \right\}.
\]
When $\mu=\Delta$ we tend to omit it from the notation, e.g. $\Omega_f(\Delta)$ is simply $\Omega_f$. The submodule of derivatives $d\Omega_f \subset \Omega_f$ is defined as the $R$-span of all $x_i\frac{\partial}
{\partial x_i} \omega$ with $\omega \in \Omega_f$ and $1 \le i \le n$.
In Part I we constructed, for every Frobenius lift $\sigma$ on $R$, an $R$-linear Cartier operator on the $p$-adic completions
\[
\cartier: \widehat \Omega_f(\mu) \to \widehat \Omega_{f^\sigma}(\mu).
\] 
This operator commutes with the derivations of $R$ and satisfies $\cartier \circ x_i\frac{\partial}
{\partial x_i} = p \, x_i\frac{\partial} {\partial x_i} \circ \cartier$ for $1 \le i \le n$. It is then immediate that the Cartier operator preserves $d \Omega_f$. We consider submodules
\[
U_f(\mu)=\{\omega\in\widehat\Omega_f(\mu) \;|\; \cartier^s(\omega)\is0 \; \mod{p^s \widehat \Omega_{f^{\sigma^s}}(\mu)} \ \mbox{for all $s\ge1$}\}.
\]
It follows from the above mentioned commutation relations that $d\Omega_f \cap \Omega_f(\mu) \subset U_f(\mu)$. Denote by $\mu_\Z = \mu \cap \Z^n$ the set of integral points in $\mu$. The main result of Part~I states that if the \emph{Hasse--Witt matrix}
\[
\beta_p(\mu) = \left(\text{ coeffcient of } \v x^{p \v v - \v u}\text{ in } f(\v x)^{p-1} \right)_{\v u, \v v \in \mu_\Z}
\]
is invertible then the quotient
\[
Q_f(\mu) = \widehat \Omega_f(\mu) / U_f(\mu)
\]
is a free $R$-module of rank $h=\# \mu_\Z$ where the images of
\[
\omega_{\v u} = \frac{\v x^\v u}{f(\v x)}, \qquad \v u \in \mu_\Z
\]
can be taken as a basis. In this case, for every Frobenius lift $\sigma$ and every derivation $\delta$ on $R$ we define matrices $\Lambda_\sigma, N_\delta \in R^{h \times h}$ by the conditions
\[\bal
\cartier (\omega_\v u) &\is \sum_{\v v \in \mu_\Z} (\Lambda_\sigma)_{\v u, \v v} \; \omega_{\v v}^\sigma \; \mod {U_{f^\sigma}(\mu)}, \\
\delta (\omega_\v u) &\is \sum_{\v v \in \mu_\Z} (N_\delta)_{\v u, \v v} \; \omega_{\v v} \; \mod {U_f(\mu)}.
\eal\]
One has $\Lambda_\sigma \is \beta_p(\mu) \, \mod p$ and hence $\cartier: Q_f(\mu) \to Q_{f^\sigma}(\mu)$ is invertible. In this paper we shall give explicit formulas for the matrices $\Lambda_\sigma, N_\delta$ in a number of situations. One $p$-adic approximation was already given in Part I:
\be{unit-root-p-adic-limits}\bal
&\Lambda_\sigma \is \beta_{p^s}(\mu) \cdot \sigma\left(\beta_{p^{s-1}}(\mu)\right)^{-1}  \; \mod {p^s},\\ 
& N_\delta \is \delta\left(\beta_{p^{s}}(\mu)\right)\cdot \beta_{p^s}(\mu)^{-1} \; \mod {p^s},
\eal\ee
where $\beta_m(\mu) \in R^{h \times h}$ is given by the same formula as the above 
Hasse--Witt matrix with $p$ replaced by a positive integer $m$.  

\bigskip
Let us say that a formal series $q(t) = \sum_{k \ge 0} b_k t^k \in \Z_p[[t]]$ with $b_0 = 1$ satisfies \emph{Dwork's congruences} if one has
\[
\frac{q(t)}{q(t^p)} \equiv \frac{\sum_{k= 0}^{p^s-1} b_k t^k}{ \sum_{k= 0}^{p^{s-1}-1} b_k t^{pk} } \; {\rm mod} \; p^s \Z_p[[t]] 
\]
for every $s \ge 1$. In~\cite{dwork69} Dwork proved this congruence for a class of hypergeometric series. His result was generalized in~\cite{MeV16} for the generating series of sequences
\[
b_k = \mbox{ constant term of } g(\v x)^k,
\]
where $g(\v x)$ is a multivariable Laurent polynomial such that its Newton polytope $\Delta$ contains $\v 0$ as its only internal integral point. In Sections~\ref{sec:periods},~\ref{sec:MeV-example} and \ref{sec:truncations} we shall apply our methods to give an alternative proof of the main result of~\cite{MeV16}. Namely, with $f(\v x)= 1 - t g(\v x)$ and $\mu = \Delta^\circ$ the module $Q_f(\mu)$ has rank~1 and we will see that $\Lambda_\sigma = q(t)/q(t^p)$. Dwork's congruence then follows from a $p$-adic approximation similar to~\eqref{unit-root-p-adic-limits}, where $\beta_{p^s}=\sum_{k=0}^{p^s-1} (-1)^k \binom{p^s-1}{k} b_k t^k$ are substituted with the truncations $\gamma_{p^s}=\sum_{k=0}^{p^s-1} b_k t^k$. In Section~\ref{sec:truncations} we explore the relation between truncations and \emph{periods modulo $m$} used in Part I; this relation is the key fact in our proof of Dwork's congruences. The main result of this paper is Theorem~\ref{main5}. It generalizes  Dwork's congruences to the A-hypergeometric setting. 

\bigskip
At the end of this introduction we would like to recall a detail from Part I which will be also useful for us here. When there is a vertex $\v b \in \Delta$ such that the coefficient of $f$ at $\v b$ is a unit in $R$, one can give the following description of our Cartier operator. By expanding rational functions into formal power series  supported in the cone $C(\Delta - \v b)$, we embed $\Omega_f$ into $\Omega_{\rm formal} = \{\sum_{\v k \in C(\Delta-\v b)} a_{\v k} \v x^\v k \;|\; a_{\v k} \in R\}$. The Cartier operation on formal expansions is simply given by
\[
\cartier: \sum_{\v k} a_{\v k} \v x^\v k \mapsto \sum_{\v k} a_{p \v k} \v x^\v k
\] 
and $U_f(\mu)$ coincides with the submodule of formal derivatives $\widehat \Omega_f(\mu) \cap d\Omega_{\rm formal}$, see \cite[Proposition 4.2]{DC-I}. 

\section{Periods}\label{sec:periods}

In Part I we introduced the Cartier operator as operator on infinite Laurent series.
However, the image of a rational function under the Cartier operator is again rational.
Consider the rational function $\omega=\frac{g(\v x)}{f(\v x)^k}\in\Omega_f$. We assert that the image of
$\omega$ under $\cartier$ is given by
\[
\frac{1}{p^n}\sum_{\v y:\v y^p=\v x}\frac{g(\v y)}{f(\v y)^k},
\]
where the summation is over all $\v y=(\zeta_p^{r_1}x_1^{1/p},\ldots,\zeta_p^{r_n}x_n^{1/p})$
with $0\le r_1,\ldots,r_n<p$, with $\zeta_p$ a primitive $p$-th root of unity.
This is again a rational function, but with denominator $\prod_{\v y:\v y^p=\v x}f(\v y)^k$.
Choose a vertex $\v b$ of the Newton polytope $\Delta$ of $f$ and expand in a Laurent series
with respect to $\v x^{\v b}$. The result is a Laurent series with support in the
cone $C(\Delta-\v b)$. Suppose it reads $\sum_{\v k}a_{\v k}\v x^{\v k}$. Then application
of $\cartier$ yields
\[
\cartier(\omega)= \frac{1}{p^n} \sum_\v k a_\v k \;\Bigl(\;
\sum_{r_1,\ldots,r_n=0}^{p-1}\zeta_p^{r_1k_1+\cdots+r_nk_n} \Bigr)\; \v x^{\v k/p}.
\]
The summation over 
the integers $r_1,\ldots,r_n$ yields something non-zero if and only if $p$ divides $k_i$
for $i=1,\ldots,n$. The summation value then equals $p^n$. Replacing $\v k$ by $p\v k$
then yields
\[
\cartier(\omega)=\sum_{\v k}a_{p\v k}\v x^{\v k},
\]
which is precisely the Cartier operator defined in Part I. 

There are also other ways to produce Laurent series expansions of $\omega$. This
happens in the case when $R$ has another non-archimedean valuation, let us call it
the $t$-adic valuation, and one coefficient of $f$ that dominates all the others $t$-adically.
So let us write $f = \sum_{\v w \in \Delta_\Z} v_{\v w} \v x^\v w$ and suppose that 
there exists $\v v$
such that $v_\v v$ is a unit in $R$ and $|v_{\v v}|_t>|v_\v w|_t$ for all $\v w \ne \v v$.  
We can then expand $\omega$ in a 
$t$-adically converging Laurent series via
\begin{eqnarray}\label{series-development}
\omega&=&\frac{g(\v x)}
{\left(v_\v v\v x^{\v v}+\sum_{\v w \ne \v v}v_{\v w}\v x^{\v w}\right)^k} = \frac{g(\v x) \v x^{-k \v v}}
{v_{\v v}^k\left(1+\sum_{\v w \ne \v v}(v_{\v w}/v_{\v v})\v x^{\v w-\v v}\right)^k}\\
&=&\frac{1}{v_{\v v}^k} g(\v x) \v x^{-k \v v} \sum_{r\ge 0}{-k \choose r}
\left(\sum_{\v w\ne \v v}(v_\v w/v_\v v)\v x^{\v w-\v v}\right)^r.
\end{eqnarray}
The series expansion is $t$-adically convergent, but when $\v v$ is not a vertex of $\Delta$
we may end up with a Laurent series
in $\v x$ whose support is not a cone. It could possibly be all of $\Z^n$.
The coefficients are then
in the completion of $R$ with respect to $|.|_t$. We denote this completion by $S$ and assume that $v_{\v v} \in S^\times$. Suppose we get
\[
\omega=\sum_{\v k\in\Z^n}c_{\v k}\v x^{\v k}, \quad c_\v k \in S.
\]
Assuming that for $v_1,v_2 \in R$ inequality $|v_1|_t > |v_2|_t$ implies 
$|\sigma(v_1)|_t > |\sigma(v_2)|_t$, one can do analogous expansion in 
$\Omega_{f^\sigma}$. Then, the same argument as above yields
\[
\cartier(\omega)=\sum_{\v k\in\Z^n}c_{p\v k}\v x^{\v k}.
\]

\begin{definition}\label{p-v-period-map}
Let $\v v\in\Delta_\Z$ be such that $|v_{\v v}|_t>|v_{\v w}|_t$ for all $\v w\in\Delta$
distinct from $\v v$ and $v_\v v \in S^\times$. Then define the period map $p_{\v v}:\Omega_f\to S$ given by $p_{\v v}(\omega) = c_\v 0$, the constant term in the Laurent series expansion of $\omega$
with respect to $\v v$.
\end{definition}

For a differential ring $S$ with a homomorphism $R \to S$ which extends the derivations of $R$,
 a \emph{period map} is an $R$-linear map $p: \Omega_f \to S$ which vanishes on 
 $d \Omega_f$ and commutes with derivations of $R$. Values of a period map on elements 
 of $\Omega_f$ are called \emph{periods}. All period maps considered in this paper satisfy 
 an extra condition of vanishing on the submodule of formal derivatives
 $U_f = \Omega_f \cap d\Omega_{\rm formal}$.

It follows almost from the definition that $p_{\v v}$ vanishes on $d\Omega_f$. 
It is slightly less trivial to see that $p_{\v v}$ vanishes on the formal derivatives.

\begin{proposition}\label{vanish-on-Uf}
Let notation be as above. Then, for all $\eta\in U_f$ we have $p_{\v v}(\eta)=0$.    
\end{proposition}

\begin{proof} First of all notice that the constant term of $\eta$ equals the constant
term of $\cartier^s(\eta)$ for all $s\ge0$. Since $\eta\in U_f$ we also know
that the $\cartier^s(\eta)\is0\mod{p^s}$. In particular the constant term of
$\eta$ is divisible by $p^s$ for all $s\ge0$, hence equals $0$. We
conclude that $p_{\v v}(\eta)=0$.
\end{proof}

\begin{theorem}\label{main3}  Let $\mu \subseteq \Delta$ be an \emph{open} set 
and $h=\# \mu_\Z$. 
Consider the column vector $\v p_{\v v} \in S^h$ with components $p_{\v v}(\omega_{\v u})$
for $\v u\in\mu_\Z$.

Assume that $R$ is $p$-adically complete and the Hasse--Witt matrix $\beta_p(\mu)$
is invertible
in $R$. For any Frobenius lift $\sigma$ and any derivation $\delta$ of $R$, 
we have 
\be{frobenius}
\v p_{\v v}=\Lambda_\sigma \; \sigma(\v p_{\v v})
\ee
and
\be{horizontality}
\delta(\v p_{\v v})= N_\delta \; \v p_{\v v}.
\ee

\end{theorem}

\begin{proof} Consider the equality
\[
\cartier(\omega_{\v u})=\sum_{\v w\in\mu_\Z}\lambda_{\v u,\v w}\omega_{\v w}^\sigma
\mod{U_f(\mu)}.
\]
Expand all terms in a Laurent series with respect to the vertex $\v v$ and determine
the constant coefficient. Using the fact that the constant term of elements in
$U_f$ vanish (Proposition~\ref{vanish-on-Uf}) we get the first statement.
In a similar vein, starting with
\[
\delta(\omega_{\v u})\is\sum_{\v w\in\mu_\Z}\nu_{\v u,\v w}\omega_{\v w}\mod{U_f(\mu)} 
\]
we get the second statement again by taking the constant term of the Laurent series
expansions with respect to $\v v$.
\end{proof}

\section{Example}\label{sec:MeV-example}
Let $g(\v x)$ be a Laurent polynomial in $x_1,\ldots,x_n$ with coefficients in $\Z_p$.
Suppose that $\v 0$ is the only lattice point in the interior of the Newton
polytope $\Delta$ of $g$.
We introduce another variable $t$ and define $f(\v x)=1-tg(\v x)$.
We apply Theorem \ref{main3} to $f(\v x)$ with $\mu=\Delta^\circ$ and $\v u=\v v=\v 0$.
In this case $\beta_m$ has only one entry, the constant coefficient of $f(\v x)^{m-1}$.
Let $R = \Z_p[t,\beta_p(t)^{-1}]\;\widehat{}\;$ be the $p$-adic completion of
$\Z_p[t,\beta_p(t)^{-1}]$. The $t$-adic closure of $R$ is $S=\Z_p[[t]]$. 
The period 
\[
q(t):= p_{\v 0}\left( \tfrac{1}{f(\v x)}\right)
\] reads $\sum_{k\ge0}b_kt^k$ with $b_k$ equal to the constant term
of $g(\v x)^k$. Take the Frobenius lift given by $t\mapsto t^p$.
Then we obtain as a consequence of Theorem \ref{main3},

\begin{corollary}\label{cor-main3}
We have 
$
\frac{q(t)}{q(t^p)}=\Lambda$
where $\Lambda \in \Z[t,\beta_p(t)^{-1}]\;\widehat{}$ is the (single entry) matrix of the Cartier operation $\cartier: Q_f(\Delta^\circ) \to Q_{f^\sigma}(\Delta^\circ)$.
\end{corollary}

One easily checks that
\[
\beta_m(t)=\sum_{k=0}^{m-1}(-1)^k{m-1\choose k}b_kt^k.
\]
Define
\[
\gamma_m(t)=\sum_{k=0}^{m-1}b_kt^k.
\]
These can be interpreted as truncated version of the power series $q(t)$. 
In \cite{MeV16} it is shown that
\begin{theorem}[Mellit--Vlasenko, 2016]\label{MeV}
For all $s\ge1$ we have $\frac{q(t)}{q(t^p)} \equiv  \frac{\gamma_{p^s}(t)}{\gamma_{p^{s-1}}(t^p)} \mod{p^s}$. 
\end{theorem}

Note that Theorem \ref{MeV} with $\gamma_m$ replaced by $\beta_m$
is simply Corollary \ref{cor-main3}. We shall prove Theorem \ref{MeV} in the next section. It will follow from our proof that in fact 
\be{DC-any-m}
\frac{q(t)}{q(t^p)} \is  \frac{\gamma_{m}(t)}{\gamma_{m/p}(t^p)} \mod{p^{\ord_p(m)}}
\ee
with any $m \ge 1$, and a similar congruence holds for the derivatives: 
\[
\frac{q'(t)}{q(t)} \is \frac{\gamma_{m}'(t)}{\gamma_{m}(t)} \; \mod{p^{\ord_p(m)}}.
\]
It is a curious fact that when $g(\v x)$  has coefficients in $\Z$ then the series $q'(t)q(t)^{-1} \in \Z[[t]]$ is a $p$-adic analytic element for each $p$.

\begin{remark}\label{dworkremark1}
Theorem~\ref{MeV} is a generalization of the famous congruence of Dwork~\cite[(12)]{dwork63}. The latter can be obtained using $g(\v x)=\frac14(x+1/x)(y+1/y)$. In `$p$-adic cycles' Dwork also proved a generalization of his congruence for a class of hypergeometric functions (see~\cite[\S 1, Corollary 2 and \S 2, Theorem 2]{dwork69}).

In that particular 
case the constant term of $g(\v x)^k$ equals ${k\choose k/2}^2 4^{-k}$ if $k$ is even and $0$
if $k$ is odd. Thus we get 
$$q(t)=\sum_{k\ge0}{2k\choose k}^2 (t/4)^{2k}=F(1/2,1/2,1|t^2).$$
Application of Theorem \ref{main3} and Corollary~\ref{cor-main3} now shows that $F(t^2)/F(t^{2p})$, hence
$F(t)/F(t^p)$, is a $p$-adic analytic element.
Here $F(t)$ is the hypergeometric function $F(1/2,1/2,1|t)$. One can put $m=2 p^s$ in~\eqref{DC-any-m} to obtain congruence~\eqref{DC-original} mentioned in the Introduction.
\end{remark}

\section{Truncations}\label{sec:truncations}

In this section we consider periods mod $m$ which, in a number of relevant
cases, turn out to be truncations of the Laurent series solutions of a system of
linear differential equations. But first we turn to general $f(\v x)$ with coefficients in
a $p$-adic ring $R$.

By a \emph{period map mod $m$} we mean an $R$-linear map $\rho: \Omega_f \to R$ such that $\rho(d\Omega_f) \subset m R$ and $\rho \circ \delta \equiv \delta \circ \rho \mod {m R}$ for every derivation $\delta$ on $R$. 
 All period maps mod $m$ considered in this paper will satisfy the condition $\rho(U_f) \subset mR$ of "vanishing" on formal derivatives.   

Choose a vertex $\v b\in\Delta$ and consider Laurent series expansions with respect 
to $\v b$. We assume its coefficient $f_{\v b}$ in $f$ to be a unit in $R$.
For an integer $m \ge 1$ and a Laurent polynomial $g(\v x) \in R[x_1^{\pm},\ldots,x_n^\pm]$
the functional
\[
\rho_{m,g}: \omega\mapsto \mbox{constant term of } g(\v x)^m \omega.
\]
is a period map mod $m$. It is clear that on formal derivatives we also
have  $\rho_{m,g}(U_f) \subset mR$. These properties follow easily if one observes that,
modulo $m$, $m$th powers behave like constants under derivations 
(see Part I, Lemma 5.1).  In Part I we already used two particular instances of these period
maps: $\tau_{m \v v} = \rho_{m,\v x^{-\v v} f(\v x)}$ for $\v v \in \Delta_\Z$ and 
$\alpha_{m\v k} = \rho_{m,\v x^{-\v k}}$ for $\v k \in C(\Delta - \v b)_\Z$. 
We now describe their behaviour under the Cartier operator and relevant congruences 
in this more general context:

\begin{proposition}
For a Laurent polynomial $g = \sum g_\v w \v x^\v w$ denote $g^\sigma = \sum g_\v w^\sigma \v x^\v w$. For any $m \ge 1$ divisible by $p$ we have $\rho_{m,g}\is \rho_{m/p,g^\sigma} \circ\cartier \quad \mod{p^{\ord_p(m)}}$.
\end{proposition}
\begin{proof}Similar to the proof of Proposition~5.2 in Part I.\end{proof}

\begin{theorem}\label{main35} Let $\mu \subseteq \Delta$ be an \emph{open} set and $h=\# \mu_\Z$. For $m \ge 1$ consider column vectors $\bm{\rho}_m \in R^h$ with components $\rho_{m,g}(\omega_{\v u})$ for $\v u\in\mu_\Z$. If $R$ is $p$-adically complete and the Hasse--Witt matrix $\beta_p(\mu)$ is invertible, then for any Frobenius lift $\sigma$ and any derivation $\delta$ of $R$ we have
\be{frobenius}
\bm{\rho}_{m} \is \Lambda_\sigma \; \sigma(\bm{\rho}_{m/p}) \quad \mod {p^{\ord_p(m)}}
\ee
and
\be{horizontality}
\delta(\bm{\rho}_{m}) \is N_\delta \; \bm{\rho}_{m} \quad \mod {p^{\ord_p(m)}}
\ee
for all $m \ge 1$.
\end{theorem}
\begin{proof}Similar to the proof of Theorem~5.3 in Part I.\end{proof}

Let us choose a tuple of elements $\phi_\v v \in R$ for $\v v \in \Delta_\Z$ and
consider matrices of periods mod $m$ given by
\be{gamma-matrix}
(\gamma_m)_{\v u,\v v \in \Delta_\Z} = \mbox{constant term of }
\left(\phi_{\v v}^m-\left(\phi_{\v v}-f(\v x)/\v x^{\v v}\right)^m\right)\omega_\v u.
\ee
Observe that the entries of $\gamma_m$ do not depend on the choice of $\v b$
since they are constant terms of Laurent polynomials that are independent of $\v b$.
For a subset $\mu \subset \Delta$ we denote by $\gamma_m(\mu)$ the submatrix given by $(\gamma_m)_{\v u,\v v \in \mu_\Z}$. We can rewrite these matrices via $\beta$-matrices as

\[
(\gamma_m)_{\v u,\v v} = \sum_{k=1}^m (-1)^{k+1} {m \choose k} \phi_{\v v}^{m-k} (\beta_{k})_{\v u,\v v},
\]
from which the following congruence follows trivially:

\begin{lemma}\label{beta-gamma-mod-p}
We have $\beta_p(\mu)\is\gamma_p(\mu)\mod{p}$. In particular $\beta_p(\mu)$ is invertible
if and only if this holds for $\gamma_p(\mu)$. 
\end{lemma}

Application of Theorem \ref{main35} to the period map given by $\rho_{m,\phi_\v v}$
minus $\rho_{m,\phi_\v v-f/\v x^{\v v}}$ yields the following

\begin{corollary}\label{main4} Let $\gamma_m(\mu)$ be as above and suppose $\gamma_p(\mu)$ is invertible. 
Then for any Frobenius lift $\sigma$ and any derivation $\delta$ of $R$ we have
\[\bal
&\gamma_{m}(\mu)\is\Lambda_\sigma \; \sigma\left(\gamma_{m /p}(\mu)\right) \; \mod {p^{\ord_p(m)}},\\
&\delta(\gamma_{m}(\mu))\is N_\delta \; \gamma_{m}(\mu) \; \mod {p^{\ord_p(m)}}
\eal\] 
for all $m\ge1$.  
\end{corollary}

As it follows from the first congruence in this corollary, we have 
\[
\gamma_{p^s}(\mu) \is \gamma_p(\mu) \cdot \sigma \left( \gamma_p(\mu) \right) \cdot \ldots \cdot \sigma^{s-1} \left( \gamma_p(\mu) \right) \; \mod p.
\]
Hence all $\gamma_{p^s}(\mu)$ are invertible and we obtain $p$-adic limit formulas
\[
\Lambda_\sigma \is \gamma_{p^s}(\mu) \cdot \sigma\left(\gamma_{p^{s-1}}(\mu)\right)^{-1}, \quad N_\delta \is \delta(\gamma_{p^s}(\mu))\cdot \gamma_{p^s}(\mu)^{-1} \; \mod {p^s}.
\]

\begin{proof}[Proof of Theorem \ref{MeV}.]
We apply Corollary~\ref{main4} in the case $f(\v x)=1-tg(\v x),\phi=1$ and $\mu=\Delta^\circ$.
Then $\gamma_m(\mu)$ is the polynomial $\sum_{k=0}^{m-1}b_kt^k$. It follows from Corollary~\ref{main4} with $\sigma(t)=t^p$ that $\gamma_{p^s}(t)\is\Lambda\gamma_{p^{s-1}}(t^p)\mod{p^s}$ for all
$s\ge1$. Theorem \ref{MeV} then follows from Corollary \ref{cor-main3} which
says that $\Lambda=q(t)/q(t^p)$.
\end{proof}

\begin{remark}\label{dworkremark2}
Here is a small variation on the proof of Theorem \ref{MeV}. We take $t_0\in\Z_p$ and
consider $f(\v x)=1-t_0g(\v x)\in \Z_p[x_1^{\pm 1},\ldots,x_n^{\pm 1}]$. Choose again $\mu=\Delta^\circ$ and suppose that $\gamma_p(t_0)\in\Z_p^\times$. Then we find that $\lim_{s\to\infty}\gamma_{p^s}(t_0)/\gamma_{p^{s-1}}(t_0)$ equals the unit root of the zeta-function of $f=0$ (by the results in the Appendix to Part I). In the Dwork example, see Remark~\ref{dworkremark1}, this means that $F_{p^s}(t_0)/F_{p^{s-1}}(t_0)$ tends to the
unit root of the zeta function of the corresponding elliptic curve. This deviates from what
one usually sees in the literature where one takes the limit
$F_{p^s}(t_0)/F_{p^{s-1}}(t_0^p)$ and $t_0$ a Teichm\"uller lift, see
for example \cite[(6.29)]{dwork69}. In the first limit we
can take any $t_0$ in its residue class and the limit will not depend on it.
\end{remark}

\section{A-hypergeometric periods}
We continue the calculation of periods following the idea in
Section 2. Let $f(\v x)=\sum_{i=1}^Nv_i\v x^{\v a_i}$,
where the $v_i$ are independent variables. This is the A-hypergeometric
setting. Let $\Delta \subset \R^n$ be the Newton polytope of $f(\v x)$, which is now the convex hull of the set $\{ \v a_1, \ldots, \v a_N \} \subset \Z^n$.
Pick some integer exponent vector $\v u \in k\Delta$, 
expand $\v x^{\v u}f(\v x)^{-k}$ with respect to $\v a_i\in\Delta_\Z$ and take
the constant term. We get
\begin{equation}\label{expand-geometric}
p_{\v a_i}\left( \v x^{\v u} f(\v x)^{-k} \right):= \mbox{ constant term of } 
\frac{\v x^{\v u - k \v a_i}}{v_i^k} \sum_{m \ge 0} \binom{-k}{m} 
\left( \sum_{r \ne i} \frac{v_r}{v_i} \v x^{\v a_r - \v a_i }\right)^m.
\end{equation}
Before we proceed we like to make a remark which considerably simplifies our
calculation. Denote by $\tilde{\v a}_r\in\Z^{n+1}$ 
the exponent vector $\v a_r$ preceded by an extra component $1$.
We call the set $A=\{ \widetilde {\v a}_1, \ldots, \widetilde {\v a}_N  \} \subset \Z^{n+1}$ 
\emph{saturated} when 
\[
\left( \sum_{j=1}^N \R_{\ge 0} \, \widetilde {\v a}_j \right) \cap \Z^{n+1} = \sum_{j=1}^N \Z_{\ge 0} \, \widetilde {\v a}_j.
\]

When $A$ is saturated, the following Proposition can be applied to any exponent vector~$\v u$:

\begin{proposition}\label{period-map-image} For an integral point $\v u \in k\Delta$ we denote $\tilde{\v u}=(k,\v u)$. Assume that there exist $\alpha_1,\ldots,\alpha_N \in \Z_{\ge 0}$ such that $\sum_{r=1}^N\alpha_r\tilde{\v a}_r=\tilde{\v u}$. Then $p_{\v a_i}(\v x^{\v u} f(\v x)^{-k})$ is equal to the application of the differential operator $\frac{(-1)^{k-1}}{(k-1)!} \prod_{r=1}^N\partial_r^{\alpha_r}$ where $\partial_r=\frac{\partial}{\partial v_r}$ to the universal series 
\[
p_{\v a_i}(\log f):=\mbox{ constant term of }\left(\log v_i+\sum_{m \ge 1} \frac{(-1)^{m-1}}{m}
\left( \sum_{r \ne i} \frac{v_r}{v_i} \v x^{\v a_r - \v a_i }\right)^m\right).
\]  
\end{proposition}

The proof is straightforward with induction on $k$.

We proceed with the calculation of $p_{\v a_i}(\log f)$ and get
\[
\log v_i+\sum_{\v l} \frac{(-1)^{\ell_1+\cdots\vee\cdots+\ell_N-1}}{\ell_1+\cdots\vee\cdots+\ell_N}
{\ell_1+\cdots\vee\cdots+\ell_N\choose \ell_1,\ldots,\vee,\ldots,\ell_N} 
\prod_{r \ne i} (v_r/v_i)^{\ell_r},
\]
where the sum is over all non-negative $\ell_1,\ldots \vee \ldots \ell_N$, not all zero, such that
$\sum_{r \ne i} \ell_r (\v a_r- \v a_i) = \v 0$. 
Here the $\vee$ in the summation range and the sum itself means that
$\ell_i$ is to be omitted. 
Introduce $\ell_i=-\sum_{r\ne i}\ell_r$. Recall our notation $\tilde{\v a}_r=(1,\v a_r)$. 
Then the definition of $\ell_i$ sees to it that the support of the resulting Laurent
series (aside from the constant $\log v_i$) is contained in the set
\[
L_i:=\{\Bell=(\ell_1,\ldots,\ell_N)\in\Z^N|\sum_{r=1}^N \ell_r\tilde{\v a}_r=\v 0,
\mbox{$\ell_r\ge0$ if $r\ne i$}\}.
\]

In order to have a more compact notation, let us rewrite the multinomial coefficient as
\[
\frac{(-1)^{\ell_1+\cdots\vee\cdots+\ell_N-1}}{\ell_1+\cdots\vee\cdots+\ell_N}
{\ell_1+\cdots\vee\cdots+\ell_N\choose \ell_1,\ldots,\vee,\ldots,\ell_N}
=\prod_{r=1}^N\frac{1}{\Gamma^*(\ell_r+1)},
\]
where $\Gamma^*(n)$ with $n\in\Z$ is defined as $(n-1)!$ if $n\ge1$ and $(-1)^n/|n|!$ if
$n\le0$. Notice that the modified $\Gamma^*$ satisfies $\Gamma^*(n+1)=n\Gamma^*(n)$ for
all integers $n\ne0$. One also checks that $\Gamma^*(n)\Gamma^*(1-n)={\rm sign}(n)(-1)^{n-1}$
for all integers $n$. Here ${\rm sign}(n)=-1$ if $n\le0$ and $1$ if $n\ge1$.
The period now takes the shape
\begin{equation}\label{p-ai-u}
p_{\v a_i}\left(\frac{\v x^{\v u}}{f(\v x)^k}\right)=
\frac{(-1)^{k-1}}{\Gamma(k)}\prod_{r=1}^N\partial_r^{\alpha_r}
\left(\log v_i+\sum_{\Bell \in L_i^*}
\prod_{r=1}^N\frac{v_r^{\ell_r}}{\Gamma^*(\ell_r+1)}\right),
\end{equation}
where $L_i^*=L_i\setminus\{\v 0\}$. Although we do not need this in the rest of this
paper, we like to notice that
this period is a Laurent series solution of the A-hypergeometric system of
equations with A-matrix the
matrix with columns $\tilde{\v a}_1,\ldots,\tilde{\v a}_N$ and parameter vector
$-\tilde{\v u}$.

When we vary the different periods over $i$ we see that the supports of the Laurent
series also vary. Fortunately it turns out that their union also lies in a regular cone.
The following result, as well as its proof, is taken from~\cite[Prop 2.9]{AS16}.
We use a different formulation however.

\begin{lemma}\label{direct-sum-cones-lemma}
Let $L_i(\R)$ be the real positive cone generated by $L_i$
and define $L^\circ(\R)=\sum_{i=1}^NL_i(\R)$. Then $L^\circ(\R)$ is a finitely 
generated cone with $\v 0$ as a vertex.
\end{lemma}

\begin{proof}
It suffices to show the following assertion. 
Let $\Bell^{(i)}\in L_{i}$ for $i=1,\ldots,N$. 
Then $\sum_{i=1}^N\Bell^{(i)}=\v 0$ implies that $\Bell^{(i)}=\v 0$ for each $i$.

Denote the coordinates of $\Bell^{(i)}$ by $l^{(i)}_k$.
Suppose that $\Bell^{(i)}\ne\v0$. 
Then $l^{(i)}_i<0$ and $l^{(i)}_k\ge0$ for all $k\ne i$. In particular 
\[
\tilde{\v a}_i=\sum_{k\ne i}-\frac{l^{(i)}_k}{l^{(i)}_i}\tilde{\v a}_k,
\]
so we see that $\tilde{\v a}_i$ is a 
(real) positive linear combination of some other $\tilde{\v a}_k$.
Define the set 
\[
C=\{\tilde{\v a}_k|\mbox{there exists $j$ such that $l^{(j)}_k\ne0$}\}.
\]
So $C$ is the set of $\tilde{\v a}_k$ that are non-trivially involved in some
relation $\Bell^{(j)}$.
Suppose $C$ is not empty. Let $\tilde{\v a}_k$ be a vertex of the convex hull of $C$.
Suppose that $l^{(k)}_k<0$.
Then $\tilde{\v a}_k$, being a positive linear combination of other $\tilde{\v a}_j\in C$
cannot be a vertex of the convex hull of $C$.
So $l^{(k)}_k\ge0$ and fortiori, $l^{(j)}_k\ge0$ for all $j$.
Their sum should be zero, contradicting the fact that $l^{(j)}_k\ne0$ for some
values of $j$. Hence we conclude that $C$ is empty. In particular $\Bell^{(j)}=\v 0$
for all $j$.
\end{proof}

Due to Lemma~\ref{direct-sum-cones-lemma} the set of formal 
power series supported in $L^\circ = L^\circ(\R) \cap \Z^N$ is a ring.
Let us denote this ring by  
\[
\sR = \{ \sum_{\Bell \in L^\circ} b_{\Bell} \v v^{\Bell} | b_{\Bell} \in \Z \}.
\]
We will also consider the bigger ring 
\[
\sS = \sR[v_1^{\pm 1},\ldots,v_N^{\pm 1}] . 
\]
Elements of $\sS$ are power series supported in a finite number of integral translations of the cone $L^\circ$.  It follows from Proposition~\ref{period-map-image} and formula~\eqref{p-ai-u} that $p_{\v a_i}(\v x^{\v u}f(\v x)^{-k}) \in (\prod_{r=1}^N v_r^{- \alpha_r}) \sR \subset \sS$. Note that when $A$ is saturated, this argument can be applied with any $k \ge 1$ and $\v u \in k\Delta$. 
With a bit more effort one can also show that $p_{\v a_i}(\v x^{\v u}f(\v x)^{-k}) \in \sS$ for any integral $\v u \in k \Delta$ without the assumption.
In what follows we shall not assume that $A$ is a saturated set. 

We shall be interested in the $N\times N$ matrix $\Psi$ with entries
\begin{equation}\label{uniform-series}
\Psi_{ji}=p_{\v a_i}(\omega_{\v a_j})=v_j^{-1}\left(\delta_{ij}+\sum_{\Bell \in L_i^*}\ell_j
\prod_{r=1}^N\frac{v_r^{\ell_r}}{\Gamma^*(\ell_r+1)}\right).
\end{equation}
This formula follows from (\ref{p-ai-u}) with $\v u=\v a_j$ and $k=1$.  It will be convenient to work with the renormalized series $\tilde{\Psi}_{ji} := v_j\Psi_{ji}\in \sR$. Let us now consider their truncated versions. Define for any $m \ge1$
the $N\times N$-matrix $\psi_m$ with entries
\[
(\psi_m)_{ji}=\mbox{ constant term of }\left(1-\left(1-\frac{f(\v x)}{v_i\v x^{\v a_i}}
\right)^m\right)\omega_{\v a_j}. 
\]
A straightforward calculation shows that this is equal to the series development
(\ref{expand-geometric}) with $k=1,\v u=\v a_j$ summed
over $m=0,1,2,\ldots,M-1$. Further calculation along the same lines as earlier
shows that we get
\be{vj-gamma-ji}
v_j(\psi_m)_{ji}=\delta_{ij}+\sum_{\Bell \in L_i(m)^*}\ell_j\prod_{k=1}^N
\frac{v_k^{\ell_k}}{\Gamma^*(\ell_k+1)},
\ee
where
\[L_i(m)=\{\Bell \in\Z^N|\sum_{k=1}^N \ell_k\tilde{\v a}_k=\v 0,
\ell_k\ge0\ \mbox{for all $k\ne i$ and }\ell_i>-m\}.
\]
Comparing~\eqref{vj-gamma-ji} and~\eqref{uniform-series} one sees that $(\tilde\psi_m)_{ji} := v_j (\psi_m)_{ji} \in \sR$ is the truncation of the element $\tilde{\Psi}_{ji} = v_j \Psi_{ji} \in \sR$. Let us consider the function $|\cdot|: L^\circ \to \Z_{\ge 0}$ given by
\[
|\Bell| := \sum_{k: \ell_k > 0} \ell_k = - \sum_{k: \ell_k < 0} \ell_k \text{ for } \Bell \in L^\circ
\] 
and define truncations of elements of $\sR$ by 
\[
r = \sum_{\Bell \in L^\circ} b_{\Bell} \v v^{\Bell} \quad \rightsquigarrow \quad r(m) := \sum_{|\Bell| \le m} b_{\Bell} \v v^{\Bell}
\]
for all $m \ge 0$. With this notation, the above computation shows that $\tilde\Psi (m) = \tilde \psi_m$. Note that the constant term of $\tilde\Psi$ is the identity matrix, and hence $\tilde\Psi$ and all its truncations $\tilde\psi_m$ are invertible over $\sR$.

\begin{theorem}\label{main5} Let $\mu \subseteq \Delta$ be an \emph{open} set  and denote $h=\#\mu_\Z$. Assume that $h \ge 1$ and $\#\{ j : \v a_j \in \mu \} = h$. Consider the $h \times h$ submatrices with entries in $\sR$ given by
\[
\tilde \Psi = (\tilde \Psi_{ji})_{\v a_j, \v a_i \in \mu},  
\]
where $\tilde\Psi_{ji}= v_j \Psi_{ji}$ are renormalized series~\eqref{uniform-series}. Let $\tilde\psi_m = \tilde \Psi(m)$ for $m \ge 1$ be the respective truncations. For the Frobenius lift $\sigma: \sR \to \sR$ that sends $v_j$ to $v_j^p$ for each $1 \le j \le N$ and any of the derivations $\delta = v_i \frac{\partial}{\partial v_i} : \sR \to \sR$ one has congruences   
\be{sigma-A-hyp-lim}
\tilde \Psi \cdot \sigma(\tilde \Psi)^{-1} \is \tilde \psi_{m} \cdot \sigma(\tilde \psi_{m/p})^{-1} \; \mod {p^{\ord_p(m)}}
\ee
and
\be{delta-A-hyp-lim}
\delta(\tilde \Psi) \cdot \tilde \Psi^{-1} \is  \delta(\tilde\psi_{m }) \cdot \tilde \psi_{m}^{-1} \quad \mod {p^{\ord_p(m)}} 
\ee
for all $m  \ge 1$.
\end{theorem}

Let $V$ be the $h \times h$ diagonal matrix with the entries $v_j$ for $\v a_j \in \mu$. Note that substituting $\tilde\Psi = V \Psi$ and $\tilde\psi_m = V \psi_m$ into~\eqref{sigma-A-hyp-lim} and~\eqref {delta-A-hyp-lim} shows that these congruences are equivalent to
\[\bal
\Psi \cdot \sigma(\Psi)^{-1} &\is \psi_{m} \cdot \sigma(\psi_{m/p})^{-1} \; \mod {p^{\ord_p(m)}},\\
\delta(\Psi) \cdot \Psi^{-1} &\is  \delta(\psi_{m}) \cdot \psi_{m}^{-1} \quad \mod {p^{\ord_p(m)}}.
\eal\]
Matrices in the latter congruences have entries in the bigger ring~$\sS$. We preferred to state our theorem for the normalized matrices because truncations are more naturally defined on elements of $\sR$ rather than $\sS$. 

\begin{proof} Consider the matrices of periods mod $m$ given by~\eqref{gamma-matrix} with $\phi_{\v a_i} = v_i$:
\be{A-hyp-gamma}
(\gamma_m)_{j,i} = \text{ constant term of } (v_i^m - (v_i - f(\v x)/\v x^{\v a_i})^m) \frac{\v x^{\v a_j}}{f(\v x)} = v_i^m (\psi_m)_{ji}.
\ee
Their entries are in $\Z[v_1,\ldots,v_N]$ and we have $\gamma_m = V^{-1} \tilde\psi_m V^m$. It particular, the coefficient of the monomial $(\prod_{\v a_j \in \mu} v_j)^{p-1}$ in $\det(\gamma_p)$ is $1$. Let $R$ be the $p$-adic completion of $\Z[v_1^{\pm1},\ldots,v_N^{\pm 1},\det(\gamma_p)^{-1}]$. Since $\det(\gamma_p)$ is not divisible by $p$, this ring satisfies our assumption $\cap_{s \ge 1} p^s R = \{0\}$ and hence one can apply Corollary~\ref{main4}. It follows that there are matrices $\Lambda_\sigma, N_\delta \in R^{h \times h}$ such that
\be{corollary-8-A-hyp}
\gamma_{m}\is \Lambda_\sigma \sigma(\gamma_{m/p})\quad \mbox{and}
\quad \delta(\gamma_{m})\is N_\delta \gamma_{m}\quad \mod{p^{\ord_p(m)}}.
\ee
Observe that all matrices $\gamma_m$ are invertible over $\sS$ because 
\[
\det(\gamma_m)=(\prod_{\v a_j \in \mu} v_j)^{m-1} \det(\tilde \psi_m) \in  (\prod_{\v a_j \in \mu} v_j)^{m-1} \sR^\times \subset \sS^\times.
\]
One of the consequences of this fact is that $R$ is a subring of the $p$-adic completion
\[
S := \widehat \sS \subset \Z_p[[v_1^{\pm 1},\ldots,v_N^{\pm 1}]].
\]
Working in the big ring $S$ we can invert matrices in~\eqref{corollary-8-A-hyp} and conclude that
\[
\gamma_{m} \cdot \sigma(\gamma_{m/p})^{-1} \is \Lambda_\sigma \quad \mbox{and}
\quad \delta(\gamma_{m})\cdot \gamma_{m}^{-1}\is N_\delta \quad \mod{p^{\ord_p(m)}}.
\]  
Substituting $\gamma_m = V^{-1} \tilde\psi_m V^m$ in the left-hand sides yields
\be{normalized-A-hyp-congruences}\bal
\tilde\psi_{m} \cdot \sigma(\tilde \psi _{m/p})^{-1} &\is V \Lambda_\sigma V^{-p} \;\mod{p^{\ord_p(m)}}\\
\quad \delta(\tilde \psi_{m})\cdot \tilde \psi_{m}^{-1} &\is V N_\delta V^{-1}+ \delta(V) V^{-1} \;\mod{p^{\ord_p(m)}}.
&\eal\ee  
One particular consequence of these congruences is that the matrices in their right-hand sides have entries in $\sR$. Secondly, they must coincide with the limits of the left-hand sides which, using the fact that $\tilde\psi_m$ is a truncation of $\tilde \Psi$, immediately implies that 
\be{normalized-A-hyp-limits}
V \Lambda_\sigma V^{-p} = \tilde \Psi \cdot \sigma(\tilde \Psi)^{-1} \quad \mbox{ and } \quad V N_\delta V^{-1}+ \delta(V) V^{-1} = \delta(\tilde \Psi) \cdot \tilde \Psi^{-1}.
\ee
Substituting these values back into~\eqref{normalized-A-hyp-congruences} proves our theorem. 
\end{proof}

The above proof is based on the ideas from Section~\ref{sec:truncations}.
By Lemma~\ref{beta-gamma-mod-p} the Hasse--Witt matrix $\beta_p(\mu)$ is congruent modulo $p$ to the matrix $\gamma_p$ given in~\eqref{A-hyp-gamma}. (In the special case $\mu=\Delta^\circ$ this was observed in~\cite[Proposition 3.8]{AS16}.) Using this fact we can conclude from the above proof that under the assumptions of Theorem~\ref{main5} the determinant of the Hasse--Witt matrix is a polynomial not divisible by $p$ and there exist the respective matrices $\Lambda_\sigma, N_\delta \in R^{h \times h}$, where $R$ is the $p$-adic completion of the ring $\Z[v_1^{\pm1},\ldots,v_N^{\pm 1},\det(\beta_p(\mu))^{-1}]$. These are the same ring $R$ and the same matrices that were used in the proof. In particular, $R$ is a subring of the $p$-adic completion $S=\widehat{\sS}$ and we have

\begin{corollary}\label{A-hyp-limits-corollary} $\Lambda_\sigma = \Psi \cdot \sigma(\Psi)^{-1}$, $N_\delta = \delta(\Psi) \cdot  \Psi^{-1}$.    
\end{corollary}
\begin{proof} Substitute $\tilde\Psi = V \Psi$ into~\eqref{normalized-A-hyp-limits}. 
\end{proof}

A special consequence of this corollary is that the matrices $V\Lambda_\sigma V^{-p}$ and $VN_\delta V^{-1}
+\delta(V) V^{-1}$ have their entries in $\sR$. Furthermore, it turns out that $N_\delta$
and, in a lesser way, $\Lambda_\sigma$, are independent of the choice of $p$.

Finally, we remark that in fact there are well defined period maps
\[
p_{\v a_i}: \widehat\Omega_f \to S.
\]
As we explained in Section~\ref{sec:periods}, these period maps are invariant under the Cartier operator (we have $p_{\v a_i} = p^\sigma_{\v a_i} \circ \cartier$ where $p^\sigma_{\v a_i}$ denotes the respective period map $\widehat{\Omega}_{f^\sigma} \to S$ ) and vanish on formal derivatives. Corollary~\ref{A-hyp-limits-corollary} is then a direct consequence of Theorem~\ref{main3}. 

Let us also mention the main result of~\cite{AS13}, Theorem 1.4. It states that in the A-hypergeometric setting with the assumption that $\Delta$ has $\v a_0$ as its unique interior lattice point the series $\Phi(\v v)/\Phi(\v v^p)$, where $\Phi(\v v)=\Psi_{00}(v_0,\ldots,v_N)$ is the unique entry of our matrix $\Psi$ for $\mu=\Delta^\circ$, is a $p$-adic analytic element with the set of poles determined by the Hasse invariant $\beta_p(\Delta^\circ)$. Hence~\cite[Theorem 1.4]{AS13} follows from Corollary~\ref{A-hyp-limits-corollary}.

\section{Example}
We continue the example from Part I, Section 7 with
\[
f(x,y)=v_1y^2+v_2x+v_3x^3+v_4x^2+v_5xy.
\]
We determine the entries of the matrix $\tilde{\Psi}$.
The vectors $\tilde{\v a}_k$ are given by the columns of
\[
\begin{pmatrix}
1 & 1 & 1 & 1 & 1\\
0 & 1 & 3 & 2 & 1\\
2 & 0 & 0 & 0 & 1
\end{pmatrix}.
\]
The supports $L_i$ lie in the null space of this matrix which can be written as
\[
(r+2s,s,s,r,-2r-4s),\qquad r,s\in\Z.
\] 
In $L_1$ we have the inequalities $s,r,-2r-4s\ge0$. This is only possible 
when $r=s=0$. The only non-trivial series $\Psi_{j,1}$ is $v_1\Psi_{1,1}=1$.

In $L_2$ we have the inequalities $r+2s,s,r,-2r-4s\ge0$ and we find $v_2\Psi_{2,2}=1$
as non-trivial series.

In $L_3$ we again get $v_3\Psi_{3,3}$ as only non-trivial $\Psi_{j,3}$.

In $L_4$ we have the inequalities $r+2s,s,-2r-4s\ge0$. Hence $r=-2s,s\ge0$. So we get
\[
v_j\Psi_{j,4}=\delta_{j,4}-\sum_{s\ge1}m_j(s)\frac{(2s-1)!}{s!s!}(v_2v_3/v_4^2)^s,
\]
where $m_j(s)$ is the $j$-th component of $(0,s,s,-2s,0)$. The $m$-truncated version
has the extra condition $m_4(s)=-2s>-m$, hence $s<m/2$.

In $L_5$ we have the inequalities $r+2s,s,r\ge0$. So we get
\[
v_j\Psi_{j,5}=\delta_{j,5}-\sum_{r,s\ge0}m_j(r,s)\frac{(2r+4s-1)!}{(r+2s)!s!s!r!}
(v_1v_4/v_5^2)^r(v_1^2v_2v_3/v_5^4)^s,
\]
where $m_j(r,s)$ is the $j$-th component of $(r+2s,s,s,r,-2r-4s)$. The $m$-truncated
version has the extra condition $m_5(r,s)=-2r-4s>-m$, hence $r+2s<m/2$. 

If we restrict our matrix to the index set $\Delta^\circ_\Z$ a computation shows
that we get the $1\times1$-matrix with element
\[
v_5\Psi_{5,5}=\sum_{r,s\ge0}\frac{(2r+4s)!}{(r+2s)!s!s!r!}x^ry^s=\frac{1}{\sqrt{1-4x}}
F\left(1/4,3/4,1\left|\frac{64y}{(1-4x)^2}\right.\right)
\]
where $x=v_1v_4/v_5^2, y=v_1^2v_2v_3/v_5^4$. The other components $v_j\Psi_{j,5}$ are
not so easy to express in terms of one-variable hypergeometric functions, if possible at all.


\begin{thebibliography}{xx}
\bibitem{AS16} A. Adolphson, S. Sperber, \emph{A-hypergeometric series and the Hasse-Witt matrix of a hypersurface}, Finite Fields Appl. 41 (2016), 55--63 
\bibitem{AS13} A. Adolphson, S. Sperber, \emph{A-hypergeometric series associated to a lattice polytope with a unique interior lattice point}, 	arXiv:1308.4439 [math.AG]
\bibitem{DC-I} F.Beukers, M.Vlasenko, \emph{Dwork crystals I}, Int. Math. Res. Notices,
published online (2020): https://doi.org/10.1093/imrn/rnaa119
\bibitem{dwork63} B. Dwork, \emph{A deformation theory for the zeta function of a hypersurface}, in Proc. Internat. Congr. Mathematicians Stockholm (Inst. Mittag-Leffler, Djursholm, 1963), 247--259
\bibitem{dwork69} B. Dwork, \emph{p-adic cycles}, Publications Math\'ematiques de l'I.H.\'E.S. 37 (1969), 27--115
\bibitem{MeV16} A. Mellit, M. Vlasenko, \emph{Dwork's congruences for the constant terms
of powers of a Laurent polynomial}, Int. J. Number Theory 12 (2016), 313--321
\end{thebibliography}
\end{document}